\documentclass[reqno]{amsart}
\usepackage[pdfstartview={Fit},pdfpagelayout=OneColumn, colorlinks=true,pdfhighlight=/I,bookmarksnumbered=true,bookmarksopen=true,bookmarksopenlevel=2, debug=false,linkcolor=blue,citecolor=blue,urlcolor=blue,pdfpagemode=UseOutlines]{hyperref}
\usepackage{lscape}
\usepackage{dsfont}
\numberwithin{equation}{section}

\newtheorem{theorem}{Theorem}
\newtheorem{prop}[theorem]{Proposition}
\newtheorem{cor}[theorem]{Corollary}

\newcommand{\taub}{\bar{\tau}}
\newcommand{\cset}{\langle c\rangle_{\mathbb{Z}}^{\perp}}
\newcommand{\transpose}{^{\mathrm{T}}}
\newcommand{\thetaab}{\vartheta_{a,b}}
\newcommand{\Bone}{B(c_{1},\nu)}
\newcommand{\Btwo}{B(c_{2},\nu)}
\newcommand{\Qone}{Q(c_{1})}

\newcommand{\mybeta}{\beta\left(-\frac{B(c,\nu)^{2}}{Q(c)}\tau_{2}\right)}
\newcommand{\betaone}{\beta\left(-\frac{\Bone^{2}}{\Qone}\tau_{2}\right)}

\newcommand{\sgn}{\mathrm{sgn}}
\newcommand{\sgnsgn}{\big\{\sgn\Bone-\sgn\Btwo\big\}}
\newcommand{\rme}{\mathrm{e}}
\newcommand{\sltwoz}{$SL(2,\mathbb{Z})$}
\newcommand{\sumZ}{\sum_{n\in\mathbb{Z}}}

\begin{document}

\title[Modular Transformations]{Modular Transformations of Ramanujan's Tenth Order Mock Theta Functions}
\author{Wynton Moore}
\address{Enrico Fermi Institute and Department of Physics, University of Chicago, Chicago, IL 60637, USA}
\email{wyntonmoore@uchicago.edu}

\thanks{Support from the US Department of State through the Fulbright Science and Technology Award, and from NSF grant 0855039, is gratefully acknowledged.}

\begin{abstract}
The modular transformations of Ramanujan's tenth order mock theta functions are computed, beginning from Choi's Hecke-type identites and using Zwegers' results on indefinite theta series. Explicit completions and shadows are found as an intermediate step. Our result for the modular transformations confirms numerical work by Gordon and McIntosh, and a recent conjecture by Cheng, Duncan, and Harvey.
\end{abstract}

\maketitle

\section{Introduction}

Mock theta functions were introduced in 1920 by Ramanujan in his final letter to Hardy (reproduced in \cite{RamanujanLost}, p127-131). The letter contained examples of order three, five, and seven, although the concept of mock theta functions and their ``order" was not fully defined. Identities between the functions were also presented. Many authors have since written on mock theta functions, in particular deriving their modular behavior and proving many of Ramanujan's identities (see, for example, \cite{Watson1,Watson2,Andrews,GordonMcIntosh}; a review is given in \cite{GordonMcIntoshSurvey}). However a general theory was not known until the work of Zwegers in 2002 \cite{ZwegersThesis,Zwegers3rd}, and its subsequent development by numerous authors \cite{ZagierBourbaki,BringmannOno2006,BringmannOno2010,EHmock1,EHmock2,Troost,Manschot}. There have since been applications in many areas of mathematics, such as combinatorics and characters of Lie superalgebras. Some recent applications in physics will be discussed shortly.
\\
\\
The previously known mock theta functions can now be characterized as examples of more general objects called mock modular forms. A \textit{mock modular form} of weight $k$ \cite{ZagierBourbaki} is a holomorphic function $F(\tau)$ on the upper half-plane $\mathbb{H}$, which is not itself modular. It can be completed to a real-analytic modular form $H(\tau)$ of weight $k$ by addition of a particular non-holomorphic correction term,
\begin{equation*}
G(\tau)=(i/2)^{k-1}\int_{-\taub}^{\infty}dz\frac{\overline{g(-\bar{z})}}{(z+\tau)^{k}}.
\end{equation*}
The function $g(\tau)$, called the \textit{shadow} of $F(\tau)$, is a holomorphic modular form of weight $2-k$. Each mock modular form has its own shadow, so they form a pair. (An equivalent characterization is to begin with a real-analytic modular form $H$ of weight $k$, and require that $(\mathrm{Im}\tau)^{k}\partial_{\taub}H$ be an anti-holomorphic modular form of weight $2-k$. Identifying this with $\bar{g}$, the difference $F=H-G$ is then a mock modular form. So it is clear that mock modular forms are a very particular class of non-modular functions.)\\
\\
The function $H=F+G$ is called the \textit{completion}. The mock modular form $F$ may be vector-valued, in which case $H$ is vector-valued also. Then $H(\tau)$ transforms as
\begin{equation*}
H(\gamma\tau)=(c\tau+d)^{k}\nu(\gamma)H(\tau),\qquad{\,}\gamma=\begin{pmatrix}a & b\\ c & d\end{pmatrix}\in\Gamma,
\end{equation*}
where $\Gamma$ is the modular group $SL(2,\mathbb{Z})$ or some congruence subgroup. The matrix-valued function $\nu(\gamma)$ is called the multiplier system or representation. We say that $F(\tau)$ is a mock modular form of weight $k$ for $\Gamma$ with multiplier system $\nu$.\\
\\
In general $F$ and $g$ are allowed to be weakly holomorphic, meaning they may have exponential growth at cusps. Going from the mock modular form to the completion, holomorphicity is traded for modularity: $F$ is holomorphic but not quite modular, while $H$ is modular but not quite holomorphic. Mock modular forms have interpretations in terms of Lerch sums, indefinite theta series, and meromorphic Jacobi forms. See \cite{ZagierBourbaki,Folsom} for reviews. Many of the previously known mock theta functions have been shown to be vector-valued mock modular forms of weight $1/2$ for \sltwoz, whose shadows are weight-$3/2$ unary theta series.\\
\\
The notation above is for compatibility with Zwegers' treatment of other mock theta functions \cite{Zwegers3rd,ZwegersThesis}, and differs slightly from more recent discussions of mock modular forms in the literature. Usually the mock modular form itself is denoted by $h$, the completion by $\hat{h}$, the shadow by $g$, and its integral by $g^{*}$.\\
\\
\\
The tenth order mock theta functions were contained in Ramanujan's lost notebook \cite{RamanujanLost} rediscovered in 1976. They are
\begin{align*}
\phi(q)&=\sum_{n\geq 0}\frac{q^{n(n+1)/2}}{(q;q^{2})_{n+1}},\qquad{\,}
\psi(q)=\sum_{n\geq 0}\frac{q^{(n+1)(n+2)/2}}{(q;q^{2})_{n+1}},\\
X(q)&=\sum_{n\geq 0}\frac{(-1)^{n}q^{n^{2}}}{(-q;q)_{2n}},\qquad{\,}
\chi(q)=\sum_{n\geq 0}\frac{(-1)^{n}q^{(n+1)^{2}}}{(-q;q)_{2n+1}},
\end{align*}
where the $q$-Pochhammer symbol is
\begin{equation*}
(a;q)_{n}=(1-a)(1-aq)\cdots(1-aq^{n-1}).
\end{equation*}
Ramanujan presented  eight identities between these functions, which were proved in a series of papers by Choi \cite{Choi1,Choi2,Choi3,Choi4}, and revisited by Zwegers \cite{Zwegers10th}. Gordon and McIntosh (\cite{GordonMcIntosh}, p216; repeated in \cite{GordonMcIntoshSurvey}) found numerical evidence for the modular transformations of the tenth order mock theta functions, but their result has not been proved until now.
\\
\\
Our main result is the following theorem, which expresses the numerical result of \cite{GordonMcIntosh} in terms of mock modular forms.

\begin{theorem}\label{maintheorem}
The vectors
\begin{equation*}
F^{(1)}(\tau)=\begin{pmatrix}q^{1/10}\phi(q^{1/2})\\ q^{-1/10}\psi(q^{1/2})\\ q^{1/10}\phi(-q^{1/2})\\ q^{-1/10}\psi(-q^{1/2})\\ q^{-1/40}X(q)\\ q^{-9/40}\chi(q)\end{pmatrix},\qquad{\,}F^{(2)}(\tau)=\begin{pmatrix}\sqrt{2}q^{1/5}\phi(q)\\ \sqrt{2}q^{-1/5}\psi(q)\\ q^{-1/80}X(q^{1/2})\\ q^{-9/80}\chi(q^{1/2})\\ q^{-1/80}X(-q^{1/2})\\ q^{-9/80}\chi(-q^{1/2})\end{pmatrix},
\end{equation*}
where $q=\rme^{2\pi i\tau}$ and $\tau\in\mathbb{H}$, are each mock modular forms of weight $1/2$ under the modular group \sltwoz. Under the generators of the modular group, $S:\tau\rightarrow-1/\tau$ and $T:\tau\rightarrow\tau+1$, they transform as
\begin{equation}\label{Stheorem}
\begin{alignedat}{2}
F^{(1,2)}(-1/\tau)&=\sqrt{-i\tau}M^{(1,2)}F^{(1,2)}(\tau)+\frac{1}{\sqrt{-i\tau}}J^{(1,2)}(\pi i/\tau),\\
F^{(1,2)}(\tau +1)&=T^{(1,2)}F^{(1,2)}(\tau),
\end{alignedat}
\end{equation}
where the multiplier system is
\begin{align*}
M^{(1)}&=\frac{2}{\sqrt{5}}\begin{pmatrix} & & & & \sin\frac{2\pi}{5} & -\sin\frac{\pi}{5}\\
 & & & & \sin\frac{\pi}{5} & \sin\frac{2\pi}{5}\\
 &  & \sin\frac{2\pi}{5} & \sin\frac{\pi}{5} &  & \\
  &  & \sin\frac{\pi}{5} & -\sin\frac{2\pi}{5} &  & \\
\sin\frac{2\pi}{5} & \sin\frac{\pi}{5} &  &  &  & \\
-\sin\frac{\pi}{5} & \sin\frac{2\pi}{5} &  &  &  & \end{pmatrix},\\
M^{(2)}&=\frac{2}{\sqrt{5}}\begin{pmatrix}  &  & \sin\frac{2\pi}{5} & -\sin\frac{\pi}{5} & & \\
 &  & \sin\frac{\pi}{5} & \sin\frac{2\pi}{5} & & \\
\sin\frac{2\pi}{5} & \sin\frac{\pi}{5} & & & & \\
-\sin\frac{\pi}{5} & \sin\frac{2\pi}{5} & & & & \\
 & & & & \sin\frac{\pi}{5} & -\sin\frac{2\pi}{5}\\
 & & & & -\sin\frac{2\pi}{5} & -\sin\frac{\pi}{5}\end{pmatrix},
\end{align*}
and
\begin{equation*}
\setlength{\arraycolsep}{3pt}
\renewcommand{\arraystretch}{0.8}
T^{(1)}=\begin{pmatrix}  &  & \zeta_{10} &  &  & \\
 &  &  & \zeta_{10}^{-1} &  & \\
\zeta_{10} &  & & &  & \\
 & \zeta_{10}^{-1} &  &  &  & \\
 &  & & & \zeta_{40}^{-1} & \\
 &  & & & & \zeta_{40}^{-9} \end{pmatrix},\;\;\;
T^{(2)}=\begin{pmatrix} \zeta_{5} &  & &  &  & \\
 & \zeta_{5}^{-1} &  &  &  & \\
 &  & & & \zeta_{80}^{-1} & \\
 &  &  &  &  & \zeta_{80}^{-9} \\
 &  & \zeta_{80}^{-1} & & & \\
 &  & & \zeta_{80}^{-9} & & \end{pmatrix},
 \setlength{\arraycolsep}{6pt}
\renewcommand{\arraystretch}{1}
\end{equation*}
with $\zeta_{n}=\rme^{2\pi i/n}$. The correction terms $J^{(1,2)}(\pi i/\tau)$ in the $S$-transformation \eqref{Stheorem} can be expressed in terms of Mordell integrals
\begin{equation*}
K_{j}(\beta)=\int_{0}^{\infty}dx\rme^{-5\beta x^{2}}\frac{\cosh{j\beta x}}{\cosh{5\beta x}},\qquad{\,} L_{j}(\beta)=\int_{0}^{\infty}dx\rme^{-5\beta x^{2}}\frac{\sinh{j\beta x}}{\sinh{5\beta x}},
\end{equation*}
as
\begin{equation*}
J^{(1)}(\beta)=\sqrt{20}\begin{pmatrix} -K_{1}(\beta)\\ -K_{3}(\beta)\\ L_{1}(\beta)\\ -L_{3}(\beta)\\ L_{4}(\beta)\\ L_{2}(\beta)\end{pmatrix},\qquad{\,}J^{(2)}(\beta)=\sqrt{40}\begin{pmatrix} -\sqrt{2}K_{1}(2\beta)\\ -\sqrt{2}K_{3}(2\beta)\\  \frac{1}{2}L_{4}(\beta/2)\\ \frac{1}{2}L_{2}(\beta/2)\\ K_{9/2}(2\beta)-K_{1/2}(2\beta)\\ K_{3/2}(2\beta)+K_{7/2}(2\beta)\end{pmatrix},
\end{equation*}
and they transform as
\begin{equation*}
J^{(1,2)}(-\pi i\tau)=-(-i\tau)^{-3/2}M^{(1,2)}J^{(1,2)}(\pi i/\tau).
\end{equation*}
\end{theorem}
\noindent\\
The conventions $\zeta_{n}=\rme^{2\pi i/n}$ ($n\in\mathbb{Z}_{>0}$) and $q=\rme^{2\pi i\tau}$ will be used throughout. For comparison of Theorem \ref{maintheorem} with p216 of \cite{GordonMcIntosh}, which uses the older convention $q=\rme^{\pi i\tau}$, note that $\frac{2}{\sqrt{5}}\sin\frac{\pi}{5}=\sqrt{\frac{5-\sqrt{5}}{10}}$ and $\frac{2}{\sqrt{5}}\sin\frac{2\pi}{5}=\sqrt{\frac{5+\sqrt{5}}{10}}$. Also, the $S$-transformation of the first four components of $F^{(2)}(\tau)$ follows from the $S$-transformation of $F^{(1)}(\tau)$.\\
\\
The proof of Theorem \ref{maintheorem} begins from identites of Choi \cite{Choi1,Choi2} and follows chapter four of Zwegers' PhD thesis \cite{ZwegersThesis}. In section \ref{ThetaSec} the indefinite theta series studied in \cite{ZwegersThesis} are briefly reviewed, and used to find the completions and shadows of $F^{(1,2)}(\tau)$. In section \ref{ModularSec} the modular transformations of the completions and shadows are worked out, resulting in a proof of Theorem \ref{maintheorem}. The vectors $F^{(1,2)}(\tau)$ were chosen for the theorem because they transform irreducibly under \sltwoz. The $S$-transformation is the main result, as the $T$-transformation follows immediately from $T:\;q\rightarrow\rme^{2\pi i}q$.\\
\\
\\
Our results also confirm a recent conjecture by Cheng, Duncan, and Harvey \cite{CDH}, concerning the modular transformations of tenth order mock theta functions. Their conjecture motivated this paper, and the context for it will now be discussed. Recently a number of connections have arisen between mock theta functions and mathematical physics, beginning with the work of Eguchi and Hikami on superconformal algebras and hyperk\"{a}hler manifolds \cite{EHmock1,EHmock2,EHentropyN4,EHentropyN2}. There have been connections to non-compact elliptic genera \cite{Troost} and black hole state counting \cite{DMZ}. Finally, a series of correspondences between mock theta functions and finite groups has been discovered in the Mathieu moonshine conjecture of Eguchi, Ooguri, and Tachikawa \cite{EOT} and its extension by Cheng, Duncan, and Harvey.\\
\\
In one such example \cite{CDH} the finite group $2.L_{2}(5).2$ is related to a certain Jacobi form $\phi^{(5)}(z,\tau)$ of weight zero and index four (given in equation (3.5) of \cite{CDH}). Some McKay-Thompson series for this correspondence are conjecturally identified with tenth order mock theta functions as follows:
\begin{equation}\label{McKayThompson}
\begin{pmatrix}H_{2B}^{(5,1)}\\ H_{2B}^{(5,2)}\\ H_{2B}^{(5,3)}\\ H_{2B}^{(5,4)}\end{pmatrix}(\tau)=-2\begin{pmatrix}q^{-1/20}X(q^{2})\\ q^{-9/20}\chi(q^{2})\\ q^{-1/5}\psi(-q)\\ -q^{1/5}\phi(-q)\end{pmatrix},\qquad{\,}\begin{pmatrix}H_{4C}^{(5,1)}\\ H_{4C}^{(5,3)}\end{pmatrix}(\tau)=-2\begin{pmatrix}q^{-1/20}X(q^{2})\\ q^{-9/20}\chi(q^{2})\end{pmatrix}.
\end{equation}
Here $H_{2B,4C}^{(5,r)}$ with $r=1,2,3,4$ are the McKay-Thompson series for the $2B,4C$ conjugacy classes of $2.L_{2}(5).2$ in its correspondence with $\phi^{(5)}$. The omitted series $H_{4C}^{(5,2)}$ and $H_{4C}^{(5,4)}$ are zero.\\
\\
The correspondence between $2.L_{2}(5).2$ and $\phi^{(5)}$ requires certain modular properties of the McKay-Thompson series and leads to a conjecture for the modular transformations of tenth order mock theta functions \cite{CDH}. Specifically, the two vectors in \eqref{McKayThompson} are conjectured to be mock modular forms of weight $1/2$ under the congruence subgroups $\Gamma_{0}(2)$ and $\Gamma_{0}(4)$, respectively, of $SL(2,\mathbb{Z})$. In section \ref{ConjectureSec} it is shown that this property follows from Theorem \ref{maintheorem}.

\section{Tenth Order Mock Theta Functions as Indefinite Theta Series}\label{ThetaSec}

Our starting point is the following set of identities expressing the tenth order mock theta functions in terms of Hecke-type double sums (these identities are due to Choi (p508 of \cite{Choi1} and p192 of \cite{Choi2}) and were rephrased as follows in \cite{Zwegers10th}):
\begin{align}
\phi(q)&=\frac{1}{\sumZ (-1)^{n}q^{n^{2}}}\bigg(\sum_{r,s\geq 0}-\sum_{r,s<0}\bigg) (-1)^{r+s}q^{r^{2}+3rs+s^{2}+r+s},\label{phi}\\
\psi(q)&=-\frac{q^{2}}{\sumZ (-1)^{n}q^{n^{2}}}\bigg(\sum_{r,s\geq 0}-\sum_{r,s<0}\bigg) (-1)^{r+s}q^{r^{2}+3rs+s^{2}+3r+3s},\notag\\
X(q)&=\frac{2q^{1/8}}{\sumZ q^{(n+1/2)^{2}/2}}\bigg(\sum_{r,s\geq 0}-\sum_{r,s<0}\bigg) q^{2r^{2}+6rs+2s^{2}+r+s},\notag\\
\chi(q)&=2-\frac{2q^{9/8}}{\sumZ q^{(n+1/2)^{2}/2}}\bigg(\sum_{r,s\geq 0}-\sum_{r,s<0}\bigg) q^{2r^{2}+6rs+2s^{2}-3r-3s}.\label{chi}
\end{align}
Double sums of this kind were first systematically studied by Hecke \cite{Hecke}. They can be related to the indefinite theta series studied by Zwegers, but first a couple of modifications will be convenient. On the right side of equation \eqref{chi} for $\chi(q)$ there is a constant term $2$. This can be absorbed into the series by making the following change of limits:
\begin{equation*}
\bigg(\sum_{r,s\geq 0}-\sum_{r,s<0}\bigg)=\bigg(\sum_{r,s>0}-\sum_{r,s\leq 0}\bigg)+\bigg(\sum_{\stackrel{r\in\mathbb{Z}}{s=0}}+\sum_{\stackrel{s\in\mathbb{Z}}{r=0}}\bigg).
\end{equation*}
The resulting extra term in the numerator of \eqref{chi} is just a multiple of the denominator:
\begin{equation*}
2q^{9/8}\bigg(\sum_{\stackrel{r\in\mathbb{Z}}{s=0}}+\sum_{\stackrel{s\in\mathbb{Z}}{r=0}}\bigg) q^{2r^{2}+6rs+2s^{2}-3r-3s}=4\sumZ q^{2(n+1/4)^{2}}=2\sumZ q^{(n+1/2)^{2}/2}
\end{equation*}
(To get the final equality, split the series at right according to $n$ even or odd.) Hence the extra term from the change of limits cancels the $2$ in \eqref{chi}. The original limits can be restored by replacing $r,s\rightarrow -r,-s$.\\
\\
It will also be convenient to split the series for $\phi(q)$ and $\psi(q)$ according to $r,s$ even or odd. For instance, \eqref{phi} becomes
\begin{multline*}
\phi(q)=\frac{1}{\sumZ (-1)^{n}q^{n^{2}}}\bigg(\sum_{r,s\geq 0}-\sum_{r,s<0}\bigg)\\ q^{4r^{2}+12rs+4s^{2}}\left\{ q^{2r+2s}-q^{6r+8s+2}-q^{8r+6s+2}+q^{12r+12s+7}\right\}.
\end{multline*}
Introducing the standard theta functions
\begin{align*}
\theta_{2}(\tau)&=\theta_{2}(0;\tau)=\sumZ q^{(n+1/2)^{2}/2},\\
\theta_{3}(\tau)&=\theta_{3}(0;\tau)=\sumZ q^{n^{2}/2},\\
\theta_{4}(\tau)&=\theta_{4}(0;\tau)=\sumZ(-1)^{n}q^{n^{2}/2},
\end{align*}
we find the following series for the components of $F^{(1)}(\tau)$:
\begin{equation}\label{F1series}
\begin{alignedat}{6}
q^{1/10}\phi(\pm q^{1/2})&=\frac{q^{1/10}}{\theta_{4,3}(\tau)}\bigg(\sum_{r,s\geq 0}-\sum_{r,s<0}\bigg)\\ &\qquad{\;}q^{2r^{2}+6rs+2s^{2}}\left\{ q^{r+s}-q^{3r+4s+1}-q^{4r+3s+1}\pm q^{6r+6s+7/2}\right\},\\
q^{-1/10}\psi(\pm q^{1/2})&=\frac{q^{9/10}}{\theta_{4,3}(\tau)}\bigg(\sum_{r,s\geq 0}-\sum_{r,s<0}\bigg)\\ &\qquad{\;} q^{2r^{2}+6rs+2s^{2}}\left\{ q^{3r+3s}-q^{5r+6s+2}-q^{6r+5s+2}\pm q^{8r+8s+11/2}\right\},\\
q^{-1/40}X(q)&=\frac{2q^{1/10}}{\theta_{2}(\tau)}\bigg(\sum_{r,s\geq 0}-\sum_{r,s<0}\bigg) q^{2r^{2}+6rs+2s^{2}+r+s},\\
q^{-9/40}\chi(q)&=\frac{2q^{9/10}}{\theta_{2}(\tau)}\bigg(\sum_{r,s\geq 0}-\sum_{r,s<0}\bigg) q^{2r^{2}+6rs+2s^{2}-3r-3s}.
\end{alignedat}
\end{equation}
In the first two lines, $+$ ($-$) corresponds to $\theta_{4}$ ($\theta_{3}$). Due to the above manipulations, the series \eqref{F1series} share the same quadratic form $2r^{2}+6rs+2s^{2}$. Similar series with quadratic form $r^{2}+3rs+s^{2}$ can be obtained for the components of $F^{(2)}(\tau)$.

\subsection{Review of Zwegers' Functions}
Zwegers \cite{ZwegersThesis} developed a theory of indefinite theta series for quadratic forms of type $(m-1,1)$ restricted to positive-definite cones. The series \eqref{F1series} are of this kind with $m=2$. Consider an indefinite quadratic form $Q$ and bilinear form $B$ defined by a symmetric $m\times m$ matrix $A$ with integer coefficients:
\begin{equation*}
B(x,y)=x\transpose Ay,\qquad{\;}Q(x)=\frac{1}{2}B(x,x).
\end{equation*}
The matrix $A$ must be of type $(m-1,1)$, that is it must have $(m-1)$ positive eigenvalues and one negative eigenvalue. Now introduce two reference vectors $c_{1},c_{2}\in\mathbb{Z}^{m}$ satisfying
\begin{equation}\label{cconstraints}
Q(c_{1}),Q(c_{2}),B(c_{1},c_{2})<0.
\end{equation}
Then the series
\begin{equation}\label{sgnsgnform}
\sum_{\nu\in\mathbb{Z}^{m}+a}\sgnsgn q^{Q(\nu)}\rme^{2\pi iB(\nu,b)},
\end{equation}
with characteristics $a,b\in\mathbb{R}^{m}$, includes only vectors $\nu$ from a sublattice of the positive-definite cone of $Q(\nu)$. However the series is not modular: following the usual Poisson summation procedure, the sublattice of included vectors is not self-dual.\\
\\
Zwegers' main result on indefinite theta series was to obtain good modular properties by rounding off the step in the $\sgn$ function in a $\tau$-dependent way. Specifically, in the case that $c_{1},c_{2}$ satisfy \eqref{cconstraints}, he replaced
\begin{align}
\sgn B(c,\nu)&\rightarrow\left(\sgn B(c,\nu)\right)\left(1-\beta\left(\frac{B(c,\nu)^{2}}{-Q(c)}\tau_{2}\right)\right)\notag\\
&=E\left(\frac{B(c,\nu)}{\sqrt{-Q(c)}}\tau_{2}^{1/2}\right),\label{sgnreplace}
\end{align}
where $\tau_{2}=\mathrm{Im}\tau$, and  $E(z)$ is the error function
\begin{equation*}
E(z)=2\int_{0}^{z}\rme^{-\pi u^{2}}du,\qquad{\;}z\in\mathbb{C}.
\end{equation*}
For $w\in\mathbb{R}$, $E(w)$ is related to the complementary error function
\begin{equation*}
\beta(v)=\int_{v}^{\infty}\frac{du}{\rme^{\pi u}\sqrt{u}},\qquad{\;}v\in\mathbb{R}_{\geq 0},
\end{equation*}
by $E(w)=\sgn(w)\left(1-\beta(w^{2})\right)$.\\
\\
Making the replacement \eqref{sgnreplace} in \eqref{sgnsgnform} results in the non-holomorphic theta series (Definition 2.1 of \cite{ZwegersThesis})
\begin{equation}\label{thetadef}
\vartheta_{a,b}(\tau)=\sum_{\nu\in\mathbb{Z}^{m}+a}\left\{E\left(\frac{B(c_{1},\nu)}{\sqrt{-Q(c_{1})}}\tau_{2}^{1/2}\right)-E\left(\frac{B(c_{2},\nu)}{\sqrt{-Q(c_{2})}}\tau_{2}^{1/2}\right)\right\}q^{Q(\nu)}\rme^{2\pi iB(\nu,b)}.
\end{equation}
Zwegers showed this has modular transformations (Corollary 2.9 of \cite{ZwegersThesis})
\begin{equation}\label{thetaabtrans}
\begin{alignedat}{2}
\vartheta_{a,b}(-1/\tau)&=\frac{i}{\sqrt{-\mathrm{det}A}}(-i\tau)^{m/2}\rme^{2\pi iB(a,b)}\sum_{p\in A^{-1}\mathbb{Z}^{m}/\mathbb{Z}^{m}}\vartheta_{b+p,-a}(\tau),\\
\vartheta_{a,b}(\tau+1)&=\rme^{-2\pi iQ(a)-\pi ia\transpose\mathrm{diag}A}\vartheta_{a,b+a+\frac{1}{2}A^{-1}\mathrm{diag}A}(\tau).
\end{alignedat}
\end{equation}
Vectors $\nu$ from the negative definite cone of $Q(\nu)$ are included in $\thetaab(\tau)$, but are suppressed by the difference in error functions as $Q(\nu)$ becomes more negative or as $\tau_{2}$ becomes large. The sum over the positive definite cone in \eqref{sgnsgnform} is equal to $\thetaab$ up to a remainder term:
\begin{multline}\label{remainder}
\sum_{\nu\in\mathbb{Z}^{m}+a}\sgnsgn q^{Q(\nu)}\rme^{2\pi iB(\nu,b)}=\thetaab(\tau)\\ -\sum_{\nu\in\mathbb{Z}^{m}+a}\left\{ \left(\sgn \Bone\right)\betaone-(c_{1}\rightarrow c_{2})\right\} q^{Q(\nu)}\rme^{2\pi iB(\nu,b)}.
\end{multline}
Zwegers also showed how to deal with this remainder term. If $c\in\mathbb{Z}^{m}$ is primitive and satifies $Q(c)<0$, then (Proposition 4.3 of \cite{ZwegersThesis})
\begin{multline}\label{Posum}
\sum_{\nu\in\mathbb{Z}^{m}+a}\big(\sgn B(c,\nu)\big)\mybeta q^{Q(\nu)}\rme^{2\pi iB(\nu,b)}\\ =-\sum_{\mu\in P(c)} \theta^{\perp c}_{\mu,b}(\tau)R_{\frac{B(c,\mu)}{2Q(c)},-B(c,b)}\big(-2Q(c)\tau\big).
\end{multline}
Here, $\theta^{\perp c}_{a,b}(\tau)$ is a restricted theta series
\begin{equation*}
\theta^{\perp c}_{a,b}(\tau)=\sum_{\xi\in\cset+a^{\perp c}}q^{Q(\xi)}\rme^{2\pi iB(\xi,b^{\perp c})},
\end{equation*}
where $\perp c$ means orthogonal to $c$ with respect to the metric $B(x,y)$:
\begin{equation*}
x^{\perp c}=x-\frac{B(c,x)}{2Q(c)}c,
\end{equation*}
and $\cset$ is the set of $x\in\mathbb{Z}^{m}$ with $B(c,x)=0$. The set $P(c)$ in \eqref{Posum} consists of all $\nu\in\mathbb{Z}^{m}+a$ with $B(c,\nu)/2Q(c)\in[0,1)$, modulo $\cset$. In other words,
\begin{equation*}
\left\{\nu\in\mathbb{Z}^{m}+a\left\vert\frac{B(c,\nu)}{2Q(c)}\in[0,1)\right.\right\}=\bigsqcup_{\mu\in P(c)}\left(\cset+\mu\right),
\end{equation*}
where $\bigsqcup$ represents disjoint union. Finally, $R_{s,t}(\tau)$ in \eqref{Posum} is a non-holomorphic unary theta series:
\begin{equation*}
R_{s,t}(\tau)=\sum_{\sigma\in\mathbb{Z}+s}\sgn(\sigma)\beta(2\sigma^{2}\tau_{2})q^{-\sigma^{2}/2}\rme^{-2\pi i\sigma t},
\end{equation*}
with $s,t\in\mathbb{R}$. It may be expressed as the period integral of a holomorphic weight-3/2 unary theta series $g_{s,t}$ (Proposition 4.2 of \cite{ZwegersThesis}):
\begin{equation}\label{Rintegral}
R_{s,t}(\tau)=-i\int_{-\taub}^{i\infty}dz\frac{g_{s,-t}(z)}{\sqrt{-i(z+\tau)}},\qquad{\;}g_{s,t}(\tau)=\sum_{\sigma\in\mathbb{Z}+s}\sigma q^{\sigma^{2}/2}\rme^{2\pi i\sigma t}.
\end{equation}
The modular transformations of $g_{s,t}$ are (Proposition 1.15 of \cite{ZwegersThesis})
\begin{equation}\label{gtrans}
g_{s,t}(-1/\tau)=i\rme^{2\pi ist}(-i\tau)^{3/2}g_{t,-s}(\tau),\;\;\;\;g_{s,t}(\tau+1)=\rme^{-\pi is(s+1)}g_{s,t+s+1/2}(\tau).
\end{equation}
Substituting in \eqref{Rintegral}, one finds $R_{s,t}$ has a well-defined transformation under $T$, but not under $S$:
\begin{equation}\label{Rnonmod}
\begin{alignedat}{2}
R_{s,t}(-1/\tau)&=i\sqrt{-i\tau}\int_{-\taub}^{0}dz\frac{(-iz)^{-3/2}g_{s,-t}(-1/z)}{\sqrt{-i(z+\tau)}}\\
&=-i\rme^{-2\pi ist}\sqrt{-i\tau}\left(R_{-t,s}+i\int_{0}^{i\infty}dz\frac{g_{-t,-s}}{\sqrt{-i(z+\tau)}}\right).
\end{alignedat}
\end{equation}
Returning to the original sum over the positive-definite cone, equations \eqref{remainder} and \eqref{Posum} combine to give
\begin{multline}\label{sgnsgnthetaR}
\sum_{\nu\in\mathbb{Z}^{m}+a}\sgnsgn q^{Q(\nu)}\rme^{2\pi iB(\nu,b)}=\thetaab(\tau)\\ -\bigg\{\sum_{\mu\in P(c_{1})} \theta^{\perp c_{1}}_{\mu,b}(\tau)R_{\frac{B(c_{1},\mu)}{2Q(c_{1})},-B(c_{1},b)}\big(-2Q(c_{1})\tau\big)-(c_{1}\rightarrow c_{2})\bigg\},
\end{multline}
provided $c_{1},c_{2}$ are primitive. The non-modularity of the left side is represented on the right by the non-modularity of $R_{s,t}$ \eqref{Rnonmod}.
\\
\\
The functions introduced in this section have the following behavior under shifts of their characteristics (Corollary 2.9 and Proposition 1.15 of \cite{ZwegersThesis}):
\begin{equation}\label{charshifts}
\begin{alignedat}{3}
\vartheta_{a+\lambda,b}=\thetaab \text{  for  }\lambda\in\mathbb{Z}^{m},\qquad{\;}&g_{s+1,t}=g_{s,t},\\
\vartheta_{a,b+\mu}=\rme^{2\pi iB(a,\mu)}\vartheta_{a,b} \text{  for  }\mu\in A^{-1}\mathbb{Z}^{m},\qquad{\;}&g_{s,t+1}=\rme^{2\pi is}g_{s,t},\\
\vartheta_{-a,-b}=-\vartheta_{a,b},\qquad{\;}&g_{-s,-t}=-g_{s,t}.
\end{alignedat}
\end{equation}
The shifts of $R_{s,t}$ are identical to those of $g_{s,t}$ except for $R_{s,t+1}=\rme^{-2\pi is}R_{s,t}$.\\
\\
Only brief definitions of Zwegers' functions, specialized to the case at hand, have been given here. One may consult \cite{ZwegersThesis} for full details.

\subsection{Application to Tenth Order Mock Theta Functions}

The following result gives the completions and shadows of $F^{(1,2)}(\tau)$ in terms of functions introduced above.
\begin{prop}\label{Hgprop}
The vectors $F^{(1,2)}(\tau)$ of Theorem \ref{maintheorem} can be expressed as
\begin{equation*}
F^{(1,2)}(\tau)=H^{(1,2)}(\tau)+G^{(1,2)}(\tau),
\end{equation*}
where
\begin{equation*}
G^{(1,2)}(\tau)=-i\int_{-\taub}^{i\infty}dz\frac{g^{(1,2)}(z)}{\sqrt{-i(z+\tau)}},
\end{equation*}
and the shadow vectors\footnote{We will refer to $g^{(1,2)}(\tau)$ as the shadow, although technically the shadow is $\overline{g^{(1,2)}(-\taub)}$.} $g^{(1,2)}(\tau)$ are integrated component by component. The completions $H^{(1,2)}(\tau)$ are
\renewcommand{\arraystretch}{1.3}
\begin{equation*}
H^{(1)}(\tau)=\begin{pmatrix}\frac{1}{\theta_{4}}\frac{1}{2}\big[\vartheta_{\frac{1}{10}e,0}-\vartheta_{\frac{1}{10}e+\frac{1}{2}e_{x},0}-\vartheta_{\frac{4}{10}e,0}+\vartheta_{\frac{4}{10}e+\frac{1}{2}e_{x},0}\big]\\
\frac{1}{\theta_{4}}\frac{1}{2}\big[\vartheta_{\frac{2}{10}e,0}-\vartheta_{\frac{2}{10}e+\frac{1}{2}e_{x},0}-\vartheta_{\frac{3}{10}e,0}+\vartheta_{\frac{3}{10}e+\frac{1}{2}e_{x},0}\big]\\
\frac{1}{\theta_{3}}\frac{1}{2}\big[\vartheta_{\frac{1}{10}e,0}-\vartheta_{\frac{1}{10}e+\frac{1}{2}e_{x},0}+\vartheta_{\frac{4}{10}e,0}+\vartheta_{\frac{4}{10}e+\frac{1}{2}e_{x},0}\big]\\
\frac{1}{\theta_{3}}\frac{1}{2}\big[-\vartheta_{\frac{2}{10}e,0}-\vartheta_{\frac{2}{10}e+\frac{1}{2}e_{x},0}-\vartheta_{\frac{3}{10}e,0}+\vartheta_{\frac{3}{10}e+\frac{1}{2}e_{x},0}\big]\\
\frac{1}{\theta_{2}}\vartheta_{\frac{1}{10}e,0}\\
\frac{1}{\theta_{2}}\vartheta_{\frac{3}{10}e,0}
\end{pmatrix}(\tau),
\end{equation*}
\begin{equation*}
\renewcommand{\arraystretch}{1.5}
H^{(2)}(\tau)=\begin{pmatrix}\frac{1}{\sqrt{2}\theta_{4}(2\tau)}\zeta_{5}^{-1}\vartheta_{\frac{2}{10}e,\frac{1}{20}e}(\tau/2)\\
\frac{1}{\sqrt{2}\theta_{4}(2\tau)}\zeta_{5}^{-2}\vartheta_{\frac{4}{10}e,\frac{1}{20}e}(\tau/2)\\
\frac{1}{\theta_{2}(\tau/2)}\vartheta_{\frac{1}{10}e,0}(\tau/2)\\
\frac{1}{\theta_{2}(\tau/2)}\vartheta_{\frac{3}{10}e,0}(\tau/2)\\
\frac{1}{\theta_{2}(\frac{\tau +1}{2})}\zeta_{80}^{-3}\vartheta_{\frac{1}{10}e,\frac{1}{20}e}(\tau/2)\\
\frac{1}{\theta_{2}(\frac{\tau +1}{2})}\zeta_{80}^{21}\vartheta_{\frac{3}{10}e,\frac{1}{20}e}(\tau/2)\end{pmatrix},\renewcommand{\arraystretch}{1}\;\;\;\text{where }e=\begin{pmatrix}1\\ 1\end{pmatrix},\,e_{x}=\begin{pmatrix}1\\ 0\end{pmatrix},
\end{equation*}
with
\begin{equation*}
A=\begin{pmatrix}4 & 6\\ 6 & 4\end{pmatrix},\;\;\;c_{1}=\begin{pmatrix}-2\\ 3\end{pmatrix},\;\;\;c_{2}=\begin{pmatrix}-3\\ 2\end{pmatrix}.
\end{equation*}
The shadows $g^{(1,2)}(\tau)$ are
\begin{equation*}
g^{(1)}(\tau)=\sqrt{20}\begin{pmatrix} -g_{\frac{4}{20},0}-g_{\frac{6}{20},0}\\
-g_{\frac{2}{20},0}-g_{\frac{8}{20},0}\\
g_{\frac{4}{20},0}-g_{\frac{6}{20},0}\\
-g_{\frac{2}{20},0}+g_{\frac{8}{20},0}\\
g_{\frac{1}{20},0}-g_{\frac{9}{20},0}\\
g_{\frac{3}{20},0}-g_{\frac{7}{20},0}
\end{pmatrix}(20\tau),\;\;\;
\renewcommand{\arraystretch}{1.2}
g^{(2)}(\tau)=\sqrt{10}\begin{pmatrix}-\sqrt{2}\zeta_{5}^{-1}g_{\frac{8}{20},\frac{1}{2}}\\
\sqrt{2}\zeta_{5}^{2}g_{\frac{4}{20},\frac{1}{2}}\\
g_{\frac{1}{20},0}-g_{\frac{9}{20},0}\\
g_{\frac{3}{20},0}-g_{\frac{7}{20},0}\\
\zeta_{40}^{-1}g_{\frac{1}{20},\frac{1}{2}}-\zeta_{40}^{-9}g_{\frac{9}{20},\frac{1}{2}}\\
\zeta_{40}^{-3}g_{\frac{3}{20},\frac{1}{2}}+\zeta_{40}^{-7}g_{\frac{7}{20},\frac{1}{2}}\end{pmatrix}(10\tau).
\end{equation*}
\renewcommand{\arraystretch}{1}
\end{prop}

\begin{proof}
The proof consists of applying equation \eqref{sgnsgnthetaR} to the Hecke-type series \eqref{F1series}. Take for instance
\begin{equation*}
q^{-1/40}X(q)=\frac{2q^{1/10}}{\theta_{2}(\tau)}\bigg(\sum_{r,s\geq 0}-\sum_{r,s<0}\bigg) q^{2r^{2}+6rs+2s^{2}+r+s},
\end{equation*}
the fifth component of $F^{(1)}(\tau)$. The exponent of $q$ can be written
\begin{align*}
2r^{2}+6rs+2s^{2}+r+s+\frac{1}{10}&=\frac{1}{2}\begin{pmatrix}r+\frac{1}{10}\\ s+\frac{1}{10}\end{pmatrix}^{\mathrm{T}}\begin{pmatrix}4 & 6\\ 6 & 4\end{pmatrix}\begin{pmatrix}r+\frac{1}{10}\\ s+\frac{1}{10}\end{pmatrix}\\
&=Q\left(\begin{pmatrix}r\\ s\end{pmatrix}+\frac{1}{10}e\right).
\end{align*}
Taking $a=\frac{1}{10}e$, the difference of signs in \eqref{sgnsgnthetaR} is
\begin{align*}
\sgnsgn&=\sgn\begin{pmatrix}10 & 0\end{pmatrix}\begin{pmatrix}r+\frac{1}{10}\\ s+\frac{1}{10}\end{pmatrix}-\sgn\begin{pmatrix}0 & -10\end{pmatrix}\begin{pmatrix}r+\frac{1}{10}\\ s+\frac{1}{10}\end{pmatrix}\notag\\
&=\begin{cases}2 & \text{if }r,s\geq 0,\\
-2 & \text{if }r,s<0,\\
0 & \text{otherwise.}\end{cases}
\end{align*}
The above series for $q^{-1/40}X(q)$ can now be written
\begin{equation*}
q^{-1/40}X(q)=\frac{1}{\theta_{2}(\tau)}\sum_{\nu\in\mathbb{Z}^{2}+\frac{1}{10}e}\sgnsgn q^{Q(\nu)}\rme^{2\pi iB(\nu,0)}.
\end{equation*}
Applying \eqref{sgnsgnthetaR} and consulting the first row of Table \ref{RHStable} results in
\begin{equation*}
q^{-1/40}X(q)=\frac{1}{\theta_{2}(\tau)}\vartheta_{\frac{1}{10}e,0}(\tau)-\frac{1}{2}\left\{R_{\frac{9}{20},0}+R_{\frac{19}{20},0}-R_{\frac{1}{20},0}-R_{\frac{11}{20},0}\right\}(20\tau).
\end{equation*}
Now use the characteristic shifts \eqref{charshifts} and the integral representation \eqref{Rintegral} of $R_{s,t}$ to get
\begin{equation*}
q^{-1/40}X(q)=\frac{1}{\theta_{2}(\tau)}\vartheta_{\frac{1}{10}e,0}(\tau)-i\sqrt{20}\int_{-\taub}^{i\infty}dz\frac{g_{\frac{1}{20},0}(20z)-g_{\frac{9}{20},0}(20z)}{\sqrt{-i(z+\tau)}},
\end{equation*}
as in the proposition.\\
\\
The remaining components of $F^{(1,2)}(\tau)$ are similar. To simplify the first four components of $g^{(1)}(\tau)$, the identities
\begin{align*}
\theta_{2}(4\tau)+\theta_{3}(4\tau)&=\theta_{3}(\tau),\\
-\theta_{2}(4\tau)+\theta_{3}(4\tau)&=\theta_{4}(\tau),
\end{align*}
were used. These can be proved by splitting the series for $\theta_{3,4}(\tau)$ according to $n$ even or odd.
\end{proof}

\section{Modular Transformations}\label{ModularSec}

The modular transformations of the completions $H^{(1,2)}(\tau)$ and shadows $g^{(1,2)}(\tau)$ introduced in Proposition \ref{Hgprop} will now be computed.

\subsection{Modular Transformations of the Completion}

Recall the $S$-transformation of $\thetaab(\tau)$ as given in \eqref{thetaabtrans}:
\begin{equation}\label{thetaabStwo}
\vartheta_{a,b}(-1/\tau)=\frac{i}{\sqrt{-\mathrm{det}A}}(-i\tau)^{m/2}\rme^{2\pi iB(a,b)}\sum_{p\in A^{-1}\mathbb{Z}^{m}/\mathbb{Z}^{m}}\vartheta_{b+p,-a}(\tau).
\end{equation}
For the value of $A$ given in Proposition \ref{Hgprop}, $A^{-1}\mathbb{Z}^{2}$ is the the set of all
\begin{equation*}
A^{-1}\begin{pmatrix}m\\ n\end{pmatrix}=\frac{1}{10}\begin{pmatrix}-2m+3n\\3m-2n\end{pmatrix}\equiv \frac{1}{10}\begin{pmatrix}u\\ v\end{pmatrix}
\end{equation*}
with $m,n,u,v\in\mathbb{Z}$. Inverting,
\begin{equation*}
5\begin{pmatrix}m\\ n\end{pmatrix}=\begin{pmatrix}2u+3v\\3u+2v\end{pmatrix}.
\end{equation*}
There is a solution $(m\; n)$ for every $(u\; v)$ with both $2u+3v$ and $3u+2v$ multiples of $5$. This is automatically true if $u=v$, and all further solutions are translations of these by $(5\; 0)$ and $(0\; 5)$. When constructing $A^{-1}\mathbb{Z}^{2}/\mathbb{Z}^{2}$ for use in \eqref{thetaabStwo}, it makes no difference what set of representatives is used, due to the characteristic shift $\vartheta_{a+\lambda,b}=\thetaab$ for $\lambda\in\mathbb{Z}^{2}$. A convenient choice is
\begin{equation*}
A^{-1}\mathbb{Z}^{2}/\mathbb{Z}^{2}=\left\{\left.\frac{v}{10}e+\frac{k}{2}e_{x}\right\vert 0\leq v\leq 9;\;k=0,1\right\},
\end{equation*}
with $e$ and $e_{x}$ as in Proposition \ref{Hgprop}. Taking $b=0$, the $S$-transformation \eqref{thetaabStwo} becomes
\begin{align}
\vartheta_{\frac{u}{10}e+\frac{j}{2}e_{x},0}(-1/\tau)&=\frac{\tau}{2\sqrt{5}}\sum_{v=0}^{9}\sum_{k=0}^{1}\vartheta_{\frac{v}{10}e+\frac{k}{2}e_{x},-\frac{u}{10}e-\frac{j}{2}e_{x}}(\tau)\notag\\
&=\frac{\tau}{2\sqrt{5}}\sum_{v=0}^{9}\sum_{k=0}^{1}(-1)^{jv+ku}\zeta_{5}^{-uv}\vartheta_{\frac{v}{10}e+\frac{k}{2}e_{x},0}(\tau)\notag\\
&=(-i\tau)\frac{1}{\sqrt{5}}\sum_{v=1}^{4}\sum_{k=0}^{1}(-1)^{jv+ku}\sin\left(\frac{2\pi uv}{5}\right)\vartheta_{\frac{v}{10}e+\frac{k}{2}e_{x},0}(\tau).\label{Sfinal}
\end{align}
The characteristic shifts \eqref{charshifts} were used to restore $b=0$ in the second line, and to combine terms in the third line. Equation \eqref{Sfinal} shows that the functions $\vartheta_{a,0}$, with $a$ in the set 
\begin{equation*}
\left\{\left.\frac{v}{10}e+\frac{k}{2}e_{x}\right\vert 1\leq v\leq 4;\;k=0,1\right\},
\end{equation*}
are closed under $S$. The completion $H^{(1)}$ in Proposition \ref{Hgprop} contains exactly these $\vartheta_{a,0}$. The set cannot be reduced further by characteristic shifts or by the $O_{A}^{+}(\mathbb{Z})$ transformations described in section 2.4 of \cite{ZwegersThesis}.\\
\\
The second completion $H^{(2)}$ contains various $\thetaab(\tau/2)$. Looking back at the definition \eqref{thetadef} of $\thetaab(\tau)$, halving $\tau$ is equivalent to halving $A$ and doubling $b$. In our case $A/2$ is still integer valued. So the $S$-transformation of $\thetaab(\tau/2)$ involves
\begin{equation*}
(A/2)^{-1}\mathbb{Z}^{2}/\mathbb{Z}^{2}=(2A^{-1}\mathbb{Z}^{2})/\mathbb{Z}^{2}=\left\{0,\frac{1}{5}e,\cdots,\frac{4}{5}e\right\}.
\end{equation*}
The characteristic shifts reduce this to just $\left\{\frac{1}{5}e,\frac{2}{5}e\right\}$. The result is similar to \eqref{Sfinal}, except different values of the second characteristic $b$ get mixed. The characteristic shifts cannot always restore $b$ because $\frac{1}{20}e\notin A^{-1}\mathbb{Z}^{2}$.\\
\\
The $T$-transformation of $\thetaab$ given in \eqref{thetaabtrans} is simpler. For the $\vartheta_{a,0}(\tau)$ occurring in $H^{(1)}$, it amounts to just a phase. For the $\thetaab(\tau/2)$ in $H^{(2)}$, the $T$-transformation can change the value of $b$.\\
\\
The completions $H^{(1,2)}$ also contain various standard theta functions. If these are packaged as vectors
\begin{equation*}
\theta_{432}(\tau)=\begin{pmatrix}\theta_{4}\\ \theta_{3}\\ \theta_{2}\end{pmatrix}(\tau),\qquad{\;} \theta_{422}(\tau)=\begin{pmatrix}\sqrt{2}\theta_{4}(2\tau)\\ \theta_{2}(\tau/2)\\ \theta_{2}\left(\frac{\tau+1}{2}\right)\end{pmatrix},
\end{equation*}
then they transform as
\begin{gather*}
\theta_{432}(-1/\tau)=\sqrt{-i\tau}\begin{pmatrix}  &  & 1\\  & 1 & \\ 1 &  & \end{pmatrix}\theta_{432}(\tau),\;\;\;\theta_{432}(\tau+1)=\begin{pmatrix} & 1 & \\ 1 &  & \\ &  & \sqrt{i}\end{pmatrix}\theta_{432}(\tau),\\
\theta_{422}(-1/\tau)=\sqrt{-i\tau}\begin{pmatrix}  & 1 & \\ 1 &  & \\  &  & 1\end{pmatrix}\theta_{422}(\tau),\;\;\;\theta_{422}(\tau+1)=\begin{pmatrix} 1 &  & \\  &  & 1\\  & \sqrt{i} & \end{pmatrix}\theta_{422}(\tau).
\end{gather*}
Applying the transformations discussed here to $H^{(1,2)}$ in Proposition \ref{Hgprop} results in
\begin{equation}\label{Htrans}
H^{(1,2)}(-1/\tau)=\sqrt{-i\tau}M^{(1,2)}H^{(1,2)}(\tau),\;\;\;\;\;H^{(1,2)}(\tau +1)=T^{(1,2)}H^{(1,2)}(\tau),
\end{equation}
with $M^{(1,2)}$ and $T^{(1,2)}$ as in Theorem \ref{maintheorem}.

\subsection{Modular Transformations of the Shadow}

The $S$-transformation of $g_{s,t}(\tau)$ was given in \eqref{gtrans}:
\begin{equation*}
g_{s,t}(-1/\tau)=i\rme^{2\pi ist}(-i\tau)^{3/2}g_{t,-s}(\tau).
\end{equation*}
The first shadow $g^{(1)}$ in Proposition \ref{Hgprop} contains functions $g_{\frac{u}{20},0}(20\tau)$, whose $S$-transformation is
\begin{equation*}
g_{\frac{u}{20},0}(-20/\tau)=-i(-i\tau/20)^{3/2}g_{0,\frac{u}{20}}(\tau/20).
\end{equation*}
Now, $g_{0,\frac{u}{20}}(\tau/20)$ can be expressed again in terms of $g_{\frac{v}{20},0}(20\tau)$:
\begin{equation}\label{gtwentytau}
\begin{alignedat}{4}
g_{0,\frac{u}{20}}(\tau/20)&=\sum_{n\in\mathbb{Z}}n\zeta_{20}^{un}q^{\frac{1}{20}n^{2}/2}\\
&=\sum_{v=0}^{19}\sum_{m\in\mathbb{Z}}(20m+v)\zeta_{20}^{u(20m+v)}q^{\frac{1}{20}(20m+v)^{2}/2}\\
&=20\sum_{v=0}^{19}\zeta_{20}^{uv}g_{\frac{v}{20},0}(20\tau)\\
&=40i\sum_{v=1}^{9}\sin\frac{\pi uv}{10}g_{\frac{v}{20},0}(20\tau),
\end{alignedat}
\end{equation}
so
\begin{equation*}
g_{\frac{u}{20},0}(-20/\tau)=(-i\tau)^{3/2}\frac{1}{\sqrt{5}}\sum_{v=1}^{9}\sin\frac{\pi uv}{10}g_{\frac{v}{20},0}(20\tau).
\end{equation*}
Terms with even (odd) $v$ drop out of the combinations $g_{\frac{u}{20},0}\pm g_{\frac{10-u}{20},0}$ appearing in $g^{(1)}$.\\
\\
The second shadow $g^{(2)}$ contains various $g_{\frac{u}{20},0}(10\tau)$ and $g_{\frac{u}{20},\frac{1}{2}}(10\tau)$, which get mixed under the $S$-transformation. This can be seen by following steps \eqref{gtwentytau} for $g_{0,\frac{u}{20}}(\tau/10)$. For the last two components of $g^{(2)}$ various $\sin\frac{n\pi}{20}$ arise, and it helps to use the identities
\begin{equation*}
\sin\frac{9\pi}{20}-\sin\frac{\pi}{20}=\sqrt{2}\sin\frac{\pi}{5},\;\;\;\;\;\;\;\sin\frac{3\pi}{20}+\sin\frac{7\pi}{20}=\sqrt{2}\sin\frac{2\pi}{5}.
\end{equation*}
\\
The $T$-transformation \eqref{gtrans} again amounts to just a phase for the $g_{\frac{u}{20},0}(20\tau)$ occurring in $g^{(1)}$. It mixes the $g_{\frac{u}{20},0}(10\tau)$ and $g_{\frac{u}{20},\frac{1}{2}}(10\tau)$ in $g^{(2)}$.\\
\\
All of this results in
\begin{equation*}
g^{(1,2)}(-1/\tau)=-(-i\tau)^{3/2}M^{(1,2)}g^{(1,2)}(\tau),\;\;\;\;g^{(1,2)}(\tau+1)=\big(T^{(1,2)}\big)^{-1}g^{(1,2)}(\tau),
\end{equation*}
with $M^{(1,2)}$ and $T^{(1,2)}$ as in Theorem \ref{maintheorem}. The integrals
\begin{equation*}
G^{(1,2)}(\tau)=-i\int_{-\taub}^{i\infty}dz\frac{g^{(1,2)}(z)}{\sqrt{-i(z+\tau)}}
\end{equation*}
in Proposition \ref{Hgprop} then transform as
\begin{equation}\label{Gtrans}
\begin{alignedat}{2}
G^{(1,2)}(-1/\tau)&=\sqrt{-i\tau}M^{(1,2)}\bigg(G^{(1,2)}(\tau)+i\int_{0}^{i\infty}dz\frac{g^{(1,2)}(z)}{\sqrt{-i(z+\tau)}}\bigg),\\
G^{(1,2)}(\tau+1)&=T^{(1,2)}G^{(1,2)}(\tau).
\end{alignedat}
\end{equation}
The correction term in \eqref{Gtrans} transforms as
\begin{equation}\label{correctiontrans}
i\int_{0}^{i\infty}dz\frac{g^{(1,2)}(z)}{\sqrt{-i(z+\tau)}}\xrightarrow{S}{}-\sqrt{-i\tau}M^{(1,2)}\left(i\int_{0}^{i\infty}dz\frac{g^{(1,2)}(z)}{\sqrt{-i(z+\tau)}}\right),
\end{equation}
as it must for $G^{(1,2)}(\tau)$ to obey $S^{2}=\mathds{1}$, a defining relation of \sltwoz. This can be seen by replacing the lower terminal of integration on the left side by $i\varepsilon$, and taking the limit $\varepsilon\rightarrow 0$ after the $S$-transformation.

\subsection{Mordell Integrals}

The final piece of Theorem \ref{maintheorem} is the correction terms $J^{(1,2)}$. The following result relates them to the corrections in the $S$-transformation of $G^{(1,2)}$ \eqref{Gtrans}.
\begin{prop}\label{Mordellprop}
The correction terms $J^{(1,2)}$ in Theorem \ref{maintheorem} can be expressed as
\begin{equation*}
\frac{1}{\sqrt{-i\tau}}J^{(1,2)}(\pi i/\tau)=i\sqrt{-i\tau}\int_{0}^{i\infty}dz\frac{M^{(1,2)}g^{(1,2)}(z)}{\sqrt{-i(z+\tau)}},
\end{equation*}
with $M^{(1,2)}$ and $g^{(1,2)}$ as in Theorem \ref{maintheorem} and Proposition \ref{Hgprop}.
\end{prop}
\begin{proof}
Our proof follows Zwegers \cite{Zwegers3rd}. Both sides of the proposition are holomorphic functions of $\tau\in\mathbb{H}$, so it suffices to consider the positive imaginary line $\tau=it$ with $t\in \mathbb{R}_{>0}$. The first component of the left side for $J^{(1)}$ is
\begin{equation}\label{Kone}
-\sqrt{\frac{20}{t}}K_{1}(\pi/t)=-\sqrt{\frac{20}{t}}\int_{0}^{\infty}dx\rme^{-5\pi x^{2}/t}\frac{\cosh(\pi x/t)}{\cosh(5\pi x/t)}.
\end{equation}
We now borrow two lemmas. Lemma 1.19 of \cite{ZwegersThesis} states that for $b\in(-\frac{1}{2},\frac{1}{2})$ and $z\in\mathbb{C}$, $z\notin(\mathbb{Z}+\frac{1}{2})i$,
\begin{equation}\label{lemmanineteen}
\frac{\rme^{2\pi bz}}{\cosh\pi z}=-\frac{1}{\pi}\sum_{r\in\mathbb{Z}+\frac{1}{2}}\frac{\rme^{2\pi ir(b+1/2)}}{z-ir}.
\end{equation}
This follows from a partial fraction expansion. Lemma 1.18 of \cite{ZwegersThesis} states that for $r\in\mathbb{R}_{\neq 0}$ and $\tau\in\mathbb{H}$,
\begin{equation}\label{lemmaeighteen}
\int_{-\infty}^{\infty}dw\frac{\rme^{\pi i\tau w^{2}}}{w+ir}=-\pi r\int_{0}^{i\infty}dz\frac{\rme^{\pi ir^{2}z}}{\sqrt{-i(z+\tau)}}.
\end{equation}
Applying \eqref{lemmanineteen} to \eqref{Kone},
\begin{align*}
-\sqrt{\frac{20}{t}}K_{1}(\pi/t)&=\sqrt{\frac{20}{t}}\frac{1}{2\pi}\int_{0}^{\infty}dx\rme^{-5\pi x^{2}/t}\sum_{r\in\mathbb{Z}+\frac{1}{2}}\frac{\rme^{2\pi ir\frac{3}{5}}+\rme^{2\pi ir\frac{2}{5}}}{\frac{5x}{t}-ir}\\
&=\sqrt{\frac{t}{5}}\frac{1}{\pi}\sum_{r\in\mathbb{Z}+\frac{1}{2}}\zeta_{5}^{2r}\int_{-\infty}^{\infty}dx'\frac{\rme^{-\pi tx'^{2}/5}}{x'-ir}.
\end{align*}
Now applying \eqref{lemmaeighteen},
\begin{align*}
-\sqrt{\frac{20}{t}}K_{1}(\pi/t)&=\sqrt{\frac{t}{5}}\sum_{r\in\mathbb{Z}+\frac{1}{2}}\zeta_{5}^{2r}r\int_{0}^{i\infty}dz\frac{\rme^{\pi ir^{2}z}}{\sqrt{-i(z+it/5)}}\\
&=\frac{\sqrt{t}}{5}\int_{0}^{i\infty}dz'\frac{g_{\frac{1}{2},\frac{2}{5}}(z'/5)}{\sqrt{-i(z'+it)}}.
\end{align*}
Finally, use
\begin{equation*}
g_{\frac{1}{2},\frac{u}{5}}(\tau/10)=20i\sum_{v=0}^{4}\sin\left(\frac{\pi u(2v+1)}{5}\right)g_{\frac{2v+1}{20},0}(10\tau)
\end{equation*}
(compare with \eqref{gtwentytau}) to get
\begin{equation*}
-\sqrt{\frac{20}{t}}K_{1}(\pi/t)=4i\sqrt{t}\int_{0}^{i\infty}dz\frac{\sin\frac{2\pi}{5}[g_{\frac{1}{20},0}-g_{\frac{9}{20},0}](20z)-\sin\frac{\pi}{5}[g_{\frac{3}{20},0}-g_{\frac{7}{20},0}](20z)}{\sqrt{-i(z+it)}},
\end{equation*}
which is the first component of the proposition for $J^{(1)}$. The remaining components of $J^{(1,2)}$ are similar.
\end{proof}

\subsection{Proof of Theorem \texorpdfstring{\ref{maintheorem}}{1}}

\begin{proof}[Proof of Theorem \ref{maintheorem}]
Applying the transformations of $H^{(1,2)}$ \eqref{Htrans} and $G^{(1,2)}$ \eqref{Gtrans} to Proposition \ref{Hgprop}, rewriting the correction terms with Proposition \ref{Mordellprop}, and applying the transformation of the correction terms \eqref{correctiontrans}, Theorem \ref{maintheorem} readily follows.
\end{proof}

\section{Mock Modular Forms for Congruence Subgroups}\label{ConjectureSec}

In this section the behavior of tenth order mock theta functions under two congruence subgroups of $SL(2,\mathbb{Z})$ is considered. The first, $\Gamma_{0}(2)$, is generated by $T$ and
\begin{equation*}
V_{1}=T^{-1}ST^{-2}S.
\end{equation*}
The second congruence subgroup, $\Gamma_{0}(4)$, is itself a subgroup of $\Gamma_{0}(2)$, and is generated by $T$ and
\begin{equation*}
V_{4}=ST^{-4}S=(TV_{1})^{2}.
\end{equation*}
As discussed in the introduction, Cheng, Duncan, and Harvey \cite{CDH} recently conjectured that the two vectors in \eqref{McKayThompson} are mock modular forms of weight $1/2$ under $\Gamma_{0}(2)$ and $\Gamma_{0}(4)$ respectively. This is confirmed by the following corollary to Theorem \ref{maintheorem}.

\begin{cor}
The vectors
\begin{gather*}
\begin{pmatrix}q^{1/10}\phi(q^{1/2})\\ q^{-1/10}\psi(q^{1/2})\\ q^{1/10}\phi(-q^{1/2})\\ q^{-1/10}\psi(-q^{1/2})\end{pmatrix}, \begin{pmatrix}q^{-1/40}X(q)\\ q^{-9/40}\chi(q)\end{pmatrix},\begin{pmatrix}q^{1/5}\phi(q)\\ q^{-1/5}\psi(q)\end{pmatrix},\begin{pmatrix}q^{1/5}\phi(-q)\\ q^{-1/5}\psi(-q)\\ q^{-1/20}X(q^{2})\\ q^{-9/20}\chi(q^{2})\end{pmatrix},\\
\begin{pmatrix}q^{-1/80}X(q^{1/2})\\ q^{-9/80}\chi(q^{1/2})\\ q^{-1/80}X(-q^{1/2})\\ q^{-9/80}\chi(-q^{1/2})\end{pmatrix},\begin{pmatrix}\sqrt{2}q^{2/5}\phi(q^{2})\\ \sqrt{2}q^{-2/5}\psi(q^{2})\\ q^{-1/40}X(-q)\\ q^{-9/40}\chi(-q)\end{pmatrix},
\end{gather*}
are all mock modular forms of weight $1/2$ under the congruence subgroup $\Gamma_{0}(2)$ of $SL(2,\mathbb{Z})$. Under $\Gamma_{0}(4)$, two of them split into smaller vectors:
\begin{gather*}
\begin{pmatrix}q^{1/5}\phi(-q)\\ q^{-1/5}\psi(-q)\\ q^{-1/20}X(q^{2})\\ q^{-9/20}\chi(q^{2})\end{pmatrix}\longrightarrow\begin{pmatrix}q^{1/5}\phi(-q)\\ q^{-1/5}\psi(-q)\end{pmatrix},\begin{pmatrix}q^{-1/20}X(q^{2})\\ q^{-9/20}\chi(q^{2})\end{pmatrix},\\
\begin{pmatrix}\sqrt{2}q^{2/5}\phi(q^{2})\\ \sqrt{2}q^{-2/5}\psi(q^{2})\\ q^{-1/40}X(-q)\\ q^{-9/40}\chi(-q)\end{pmatrix}\longrightarrow\begin{pmatrix}q^{2/5}\phi(q^{2})\\ q^{-2/5}\psi(q^{2})\end{pmatrix},\begin{pmatrix}q^{-1/40}X(-q)\\ q^{-9/40}\chi(-q)\end{pmatrix}.
\end{gather*}
\end{cor}

\begin{proof}
Take for instance the third and fourth vectors of the corollary. These are the first two and last four components of $F^{(1)}(2\tau)$ respectively, so it is enough to show $F^{(1)}(2\tau)$ reduces to them under $\Gamma_{0}(2)$. The $T$-transformation of $F^{(1)}(2\tau)$ is
\begin{equation*}
F^{(1)}\big(2(\tau+1)\big)=\big(T^{(1)}\big)^{2}F^{(1)}(2\tau),
\end{equation*}
which is diagonal. As for the other generator $V_{1}$, it induces the following action on $2\tau$:
\begin{equation*}
2V_{1}\tau=2\frac{\tau+1}{-2\tau-1}=\left[T^{-2}ST^{-1}S\tau\right]_{\tau\rightarrow 2\tau}.
\end{equation*}
Now note that
\begin{equation*}
\big(T^{(1)}\big)^{-2}M^{(1)}\big(T^{(1)}\big)^{-1}M^{(1)}=\begin{pmatrix}X & & \\ & & Y\\ & Z & \end{pmatrix},
\end{equation*}
where
\begin{gather*}
\renewcommand{\arraystretch}{1.3}
X=\frac{4}{5}\begin{pmatrix}
\zeta_{40}\sin^{2}\frac{\pi}{5}+\zeta_{40}^{-7}\sin^{2}\frac{2\pi}{5} & 2\zeta_{40}^{-13}\sin^{2}\frac{\pi}{5}\sin\frac{2\pi}{5}\\
2\zeta_{40}^{3}\sin^{2}\frac{\pi}{5}\sin\frac{2\pi}{5} & \zeta_{40}^{9}\sin^{2}\frac{\pi}{5}+\zeta_{40}^{17}\sin^{2}\frac{2\pi}{5}
\end{pmatrix},\\
\renewcommand{\arraystretch}{1.3}
Y=\frac{4}{5}\begin{pmatrix}
\zeta_{10}^{-1}\sin^{2}\frac{\pi}{5}+\zeta_{10}^{-3}\sin^{2}\frac{2\pi}{5} & 2\zeta_{20}\sin^{2}\frac{\pi}{5}\sin\frac{2\pi}{5}\\
2\zeta_{20}^{-1}\sin^{2}\frac{\pi}{5}\sin\frac{2\pi}{5} & -\zeta_{10}\sin^{2}\frac{\pi}{5}-\zeta_{10}^{3}\sin^{2}\frac{2\pi}{5}\end{pmatrix},\\
\renewcommand{\arraystretch}{1.3}
Z=\frac{4}{5}\begin{pmatrix}\zeta_{20}^{3}\sin^{2}\frac{\pi}{5}+\zeta_{20}^{-1}\sin^{2}\frac{2\pi}{5} & 2\zeta_{5}^{-1}\sin^{2}\frac{\pi}{5}\sin\frac{2\pi}{5}\\
-2\zeta_{5}\sin^{2}\frac{\pi}{5}\sin\frac{2\pi}{5} & \zeta_{20}^{-3}\sin^{2}\frac{\pi}{5}+\zeta_{20}\sin^{2}\frac{2\pi}{5}
\end{pmatrix}.
\renewcommand{\arraystretch}{1}
\end{gather*}
Then the action of $V_{1}$ on $F^{(1)}(2\tau)$ is
\begin{equation*}
F^{(1)}(2V_{1}\tau)=\sqrt{2\tau+1}\begin{pmatrix}X & & \\ & & Y\\ & Z & \end{pmatrix}\left(F^{(1)}(2\tau)+i\int_{1}^{i\infty}dz\frac{g^{(1)}(z)}{\sqrt{-i(z+2\tau)}}\right).
\end{equation*}
Hence, under $\Gamma_{0}(2)$, $F^{(1)}(2\tau)$ reduces to the third and fourth vectors of the corollary, and these vectors individually are mock modular forms for $\Gamma_{0}(2)$. The remaining vectors of the corollary are the reductions of $F^{(1,2)}(\tau)$ and $F^{(1,2)}(2\tau)$ under $\Gamma_{0}(2)$ and $\Gamma_{0}(4)$.

\end{proof}

\subsection*{Acknowledgements}
I wish to thank my advisor Jeff Harvey for suggesting this problem, and for helpful discussions.

\begin{landscape}
\renewcommand{\arraystretch}{1.4}
\begin{table}
\caption{Elements of equation \eqref{sgnsgnthetaR} relevant for Proposition \ref{Hgprop}.}
\begin{center}
\begin{equation*}
\begin{array}{l||c|c|c|c|c|c|c|c}
\multicolumn{1}{c||}{a} & P(c_{1}) & \frac{B(c_{1},P(c_{1}))}{2Q(c_{1})} & \theta^{\perp c_{1}}_{P(c_{1}),0}(\tau) & \theta^{\perp c_{1}}_{P(c_{1}),\frac{1}{20}e}(\tau/2) & P(c_{2}) & \frac{B(c_{2},P(c_{2}))}{2Q(c_{2})} & \theta^{\perp c_{2}}_{P(c_{2}),0}(\tau) & \theta^{\perp c_{2}}_{P(c_{2}),\frac{1}{20}e}(\tau/2)\\
[0.4ex] \hline

\frac{1}{10}e & -\frac{9}{10}e,-\frac{19}{10}e & \frac{9}{20},\frac{19}{20} & \frac{1}{2}\theta_{2}(\tau) & \frac{1}{2}\zeta_{16}^{-3,1}\theta_{2}\left(\frac{\tau+1}{2}\right) & \frac{1}{10}e,\frac{11}{10}e & \frac{1}{20},\frac{11}{20} & \frac{1}{2}\theta_{2}(\tau) & \frac{1}{2}\zeta_{16}^{1,-3}\theta_{2}\left(\frac{\tau+1}{2}\right)\\

\frac{1}{10}e+\frac{1}{2}e_{x} & \uparrow +\frac{1}{2}e_{x} & \uparrow -\frac{5}{20} & \theta_{2,3}(4\tau) & 0,\theta_{4}(2\tau) & \uparrow +\frac{1}{2}e_{x} & \uparrow  & \frac{1}{2}\theta_{2}(\tau) & \frac{1}{2}\zeta_{16}^{-3,1}\theta_{2}\left(\frac{\tau+1}{2}\right)\\

\frac{2}{10}e & -\frac{8}{10}e,-\frac{18}{10}e & \frac{8}{20},\frac{18}{20} & \theta_{3,2}(4\tau) & \theta_{4}(2\tau),0 & \frac{2}{10}e,\frac{12}{10}e & \frac{2}{20},\frac{12}{20} & \theta_{2,3}(4\tau) & 0,\theta_{4}(2\tau)\\

\frac{2}{10}e+\frac{1}{2}e_{x} & \uparrow +\frac{1}{2}e_{x} & \uparrow -\frac{5}{20} & \frac{1}{2}\theta_{2}(\tau) & \frac{1}{2}\zeta_{16}^{-3,1}\theta_{2}\left(\frac{\tau+1}{2}\right) & \uparrow +\frac{1}{2}e_{x} & \uparrow  & \theta_{3,2}(4\tau) & \theta_{4}(2\tau),0\\

\frac{3}{10}e & -\frac{7}{10}e,-\frac{17}{10}e & \frac{7}{20},\frac{17}{20} & \frac{1}{2}\theta_{2}(\tau) & \frac{1}{2}\zeta_{16}^{1,-3}\theta_{2}\left(\frac{\tau+1}{2}\right) & \frac{3}{10}e,\frac{13}{10}e & \frac{3}{20},\frac{13}{20} & \frac{1}{2}\theta_{2}(\tau) & \frac{1}{2}\zeta_{16}^{-3,1}\theta_{2}\left(\frac{\tau+1}{2}\right)\\

\frac{3}{10}e+\frac{1}{2}e_{x} & \uparrow +\frac{1}{2}e_{x} & \uparrow -\frac{5}{20} & \theta_{3,2}(4\tau) & \theta_{4}(2\tau),0 & \uparrow +\frac{1}{2}e_{x} & \uparrow  & \frac{1}{2}\theta_{2}(\tau) & \frac{1}{2}\zeta_{16}^{1,-3}\theta_{2}\left(\frac{\tau+1}{2}\right)\\

\frac{4}{10}e & -\frac{6}{10}e,-\frac{16}{10}e & \frac{6}{20},\frac{16}{20} & \theta_{2,3}(4\tau) & 0,\theta_{4}(2\tau) & \frac{4}{10}e,\frac{14}{10}e & \frac{4}{20},\frac{14}{20} & \theta_{3,2}(4\tau) & \theta_{4}(2\tau),0\\

\frac{4}{10}e+\frac{1}{2}e_{x} & \uparrow +\frac{1}{2}e_{x} & \uparrow -\frac{5}{20} & \frac{1}{2}\theta_{2}(\tau) & \frac{1}{2}\zeta_{16}^{1,-3}\theta_{2}\left(\frac{\tau+1}{2}\right) & \uparrow +\frac{1}{2}e_{x} & \uparrow  & \theta_{2,3}(4\tau) & 0,\theta_{4}(2\tau)
\end{array}
\end{equation*}
\end{center}
\label{RHStable}
\end{table}
\renewcommand{\arraystretch}{1}
\end{landscape}

\bibliographystyle{amsplain}
\bibliography{10th_Refs}

\end{document}